\documentclass[12pt]{amsart}
\usepackage{amssymb}
\usepackage{rotating}
\usepackage{tikz}
\usepackage{cite}
\usepackage{hyperref}
\newcommand{\ad }{\mathrm{ad}\,}
\newcommand{\fie }{\Bbbk}
\newcommand{\dpu }{\hat{u}}   

\newcommand{\NA }{\mathcal{B}}
\newcommand{\ndN }{\mathbb{N}}
\newcommand{\ndZ }{\mathbb{Z}}
\newcommand{\ndJ }{\mathbb{J}}
\newcommand{\lex }{<_{\mathrm{lex}}}
\newcommand{\ot }{\otimes}

\newcommand{\su }{S}
\newcommand{\tabincell}[2]{\begin{tabular}{@{}#1@{}}#2\end{tabular}}
	
\newtheorem{theo}{Theorem}[section]
\newtheorem{prop}[theo]{Proposition}
\newtheorem{lemm}[theo]{Lemma}
\newtheorem{coro}[theo]{Corollary}

\theoremstyle{definition}
\newtheorem{defi}[theo]{Definition}
\newtheorem{exam}[theo]{Example}
\newtheorem{rema}[theo]{Remark}

\begin{document}
\title[root multiplicities over arbitrary fields]
{ On some root multiplicities for Nichols algebras of diagonal type over arbitrary fields }


\author{Ying Zheng}
\address{School of mathematical Sciences, East China Normal University, Shanghai 200241, China.}
\email{52150601007@ecnu.cn}
\thanks{The author was supported by China Scholarship Council}

\date{}

\maketitle

\begin{abstract}
   In this paper, our main aim is to determine the multiplicities of a class of roots for Nichols algebra of diagonal type over fields of arbitrary characteristic, which is a generalization of the results on the multiplicities of these roots over fields of characteristic zero obtained by I. Heckenberber and the author.

  \textit{Keywords}: {Nichols algebras, multiplicity, root vector}
\end{abstract}

\section{Introduction}
Motivated by Hopf algebra theory, W. Nichols first introduced the structure of Nichols algebra in 1978 \cite{MR0506406}. The theory of Nichols algebras has interesting applications in various research areas of mathematics. In particular, the Nichols algebras appeared naturally in the theory of pointed Hopf algebra (see for example \cite{AS98}).

Since the root systems and Weyl groupoids of Nichols algebras of diagonal type were introduced \cite{H2006}, which are based on existence of Poincare-Birkhoff-Witt basis \cite{khar}, the Nichols algebras of diagonal type received a remarkable development. One of the biggest open problems in the theory of Nichols algebras of diagonal type is to determine the root systems and multiplicity of a given root. It is well known that for finite-dimensional Nichols algebras of diagonal type their root system is finite and the multiplicity of any root is one \cite{H08}. Roots of the form $m\alpha_1+\alpha_2$ (where $\alpha_1,\alpha_2$ is the standard basis of $\ndZ^2$.) with $m\ge0$ were determined by
M. Rosso in \cite{MR1632802}. Over fields of characteristic zero, I. Heckenberber and the author have determined roots of the form $m\alpha_1+2\alpha_2$ with $m\ge0$ and their multiplicities  in \cite{H2018root}. It is natural and desirable to address the roots of the form $m\alpha_1+2\alpha_2$ with $m\in\ndN_0$ and to determine their multiplicities for arbitrary fields.

The present paper is devote to generalizing the results in \cite{H2018root} to arbitrary fields. For any $m\ge0$, the multiplicity of root $m\alpha_1+2\alpha_2$ over arbitrary fields is determined in this paper.
Let $\fie$ be a field and the characteristic of $\fie$ is $p$, where $p$ is a prime number or zero.
Let $(V,c)$ be a two-dimensional braided  vector space of diagonal type over $\fie$ and let $(q_{ij})_{1\le i,j\le2}$ be the braiding matrix of $V$. Let $T(V)$ and $\NA(V)$ denote the tensor algebra and Nichols algebra of $V$, respectively.
We define some particular elements $L_n$ with some $n\in \ndN_0$, which are some homogenous elements in $T(V)$ (for details see Lemma \ref{le:existence}). These $L_n$ with $n\in\ndN_0$ do not appear in \cite{H2018root}, in fact they can just appear over fields of positive characteristic. Let $(P_k)_{k\in\ndN_0}$ be the homogenous elements in $T(V)$ defined in \cite{H2018root}, which will be recalled in this paper see Definition \ref{de:Pk}. We relate the relations in $\NA(V)$ of degree $m\alpha_1+2\alpha_2$ to these elements $P_k, L_n$ for some $n, k\in \ndN_0$ in Theorem \ref{theo:main}.
Let $q=q_{11}, r=q_{12}q_{21},s=q_{22}$ and let
 $\pi:T(V)\rightarrow \NA(V)$ be the canonical map. We define a subset $\ndJ=\ndJ^p_{q,r,s}$ of $\ndN_0$ to measure the multiplicities of all roots of the form $m\alpha_1+2\alpha_2$. Let $\ndJ_1, \ndJ_2$ be subsets of $\ndJ$ such that $\ndJ=\ndJ_1\sqcup\ndJ_2$, which are defined in Definition \ref{def:J1J2}.  For any $k\in\ndN_0$, define
 $b_k=\prod_{i=0}^{k-1}(1-q^ir)$. Let $U_n$ with $n\in \ndN_0$ be the subspace of $T(V)$ of degree $n\alpha+2\alpha$ in \eqref{eq:Un}.
 Our main results is following.

\begin{theo}$($See Theorem \ref{theo:main}$)$
Let $m\in \ndN_0$ such that $(m)_q^!b_m\ne0$. Then the elements
$$(\ad x_1)^{m-j}(P_j), ~(\ad x_1)^{m-n}(L_n),~ \text{for} ~j\in \ndJ_1 ~\text{and}~ n\in \ndJ_2,$$
form a basis of $\ker(\pi)\cap U_m$.
\end{theo}

The paper is organized as follows. In Section \ref{se:basic}, we recall some basic notions and notations from \cite{H2018root}, which will be useful later. Some results from \cite{H2018root} are also recalled in this section.
In Section \ref{se:multi}, we illustrate and prove our main theorem mentioned above.

The paper was written during the visit of the author to Marburg University
supported by China Scholarship Council. The author is very grateful to professor Istv\'{a}n
Heckenberger for suggesting this question and for valuable comments on the draft of this paper.
The author also thank Eric Heymann-Heidelberger for useful discussions on the topic.

\section{Preliminaries}
\label{se:basic}
In this section we recall some notations and notions from \cite{H2018root} and also some results, which will be used later.

Let $\fie$ be a field and $\fie^{\times}=\fie\backslash\{0\}$. As usual we write $\ndN$ for the set of positive integers and write $\ndZ$ for the set of integers. Let $\ndN_0=\ndN\cup\{0\}$.

We start collecting some information about root vector from \cite{H2018root}.

Let $(V,c)$ be a braided vector space of rank n. Let $I=\{1,2,\ldots,n\}$, and $x_1,x_2,\ldots,x_n$ be a basis of $V$, let
$(q_{ij})_{i,j\in I} \in (\fie^{\times})^{n\times n}$ such that
$$c(x_i\ot x_j)=q_{ij}x_j\ot x_i, ~\text{for any}~i,j\in I.$$

Let $\alpha_1,\alpha_2,\ldots,\alpha_n$ be the standard basis of $\ndZ^n$ and let
$\chi: \ndZ^n \times \ndZ^n\rightarrow \fie^{\times}$ denote the bicharacter of $\ndZ^n$ given by
$\chi(\alpha_i,\alpha_j)=q_{ij}$, for any $i,j\in I$. Let $T(V)$ be the tensor algebra of $V$ and let $\NA(V)$ be the Nichols algebra of $V$. We write $\pi:T(V)\rightarrow \NA(V)$ for the canonical map.

Let $X=\{x_1,x_2,\ldots,x_n\}$. Let us fix the total ordering $<$ on $X$ such that $x_i<x_j$ if and only if
$1\le i<j\le n$. Let $\mathbb{X}$ and $\mathbb{X}^{\times}$ denote the set of words and non-empty words with letters in $X$, respectively. We write $\lex$ for the
lexicographic ordering which is induced by the total ordering $<$ on $X$ (for detail see \cite[Section~3]{H2018root}).
For any word
$w=x_{i_1}x_{i_2}\cdots x_{i_s}\in \mathbb{X}$, we write $|w|=s$ for the length of $w$.

Both $T(V)$ and $\NA(V)$ have a unique $\ndZ^n$-graded braided bialgebra structure such that $\deg(x_i)=\alpha_i$, for any $i\in I$. Specially, for any word $w=x_{j_1}x_{j_2}\cdots x_{j_t}\in \mathbb{X}$, $\deg(w)=\sum_{i=1}^t\alpha_{j_i}$. We write $\deg(x)$ for the degree of any homogeneous element $x$ of $T(V)$ or $\NA(V)$.

For basic information of Lyndon words we refer to \cite{H2018root} (or see \cite{ML97}).

For a Lyndon word $w\in \mathbb{X}^{\times}$, we define \textbf{super-letter} $[w]\in T(V)$ inductively given by:

\begin{enumerate}
 \item $[w]=w$, if $w\in X$, and
\item $[w]=[u][v]-\chi(\deg(u),\deg(v))[v][u]$
  if $|w|\ge 2$, and  $w=uv$ is the Shirshow
  decomposition of $w$.
\end{enumerate}

For any Lyndon word $w$ and any integer $k\ge2$ we write $[w^k]=[w]^k$.

We also write $\lex$ for the total ordering on the set of super-letters, which is induced by the total ordering $\lex$ on $\mathbb{X}$.

For any $\alpha \in \ndZ^n$, let $o_\alpha \in \ndN \cup \{\infty\}$
be the multiplicative order of $\chi(\alpha ,\alpha )\in \fie^\times $.
Moreover, let
\begin{align*}
	O_\alpha =\begin{cases}
		\{1,o_\alpha,\infty \} & \text{if $o_\alpha=\infty $ or
			$\mathrm{char}(\fie)=0$,}\\
		\{1,o_\alpha p^k,\infty \mid k\in \ndN_0 \} &
		\text{if $o_\alpha <\infty $, $p=\mathrm{char}(\fie )>0$.}
	\end{cases}
\end{align*}

Let $w\in\mathbb X^\times $. We say that $[w]$ is a
\textbf{root vector candidate} if
$w=v^k$ for some Lyndon word $v$
and $k\in O_{\deg v}\setminus \{\infty \}$.

A root vector candidate $[w]$, where $w\in \mathbb X^\times $,
is termed a \textbf{root vector} (of $\NA (V)$)
if $[w]\in \NA (V)$ is not a
linear combination of elements of the form
$[v_k]^{m_k}\cdots [v_1]^{m_1}$, where $k\in \ndN_0$ and
$[v_1],\dots,[v_k]$ are root vector candidates with
$w\lex v_1\lex \cdots \lex v_k$.

Let $L\subseteq \mathbb{X}^{\times}$ such that $w\in L$ whenever $[w]$ is a root vector.

Let
$$\boldsymbol{\Delta}_+ = \{\deg (u)\mid u\in L\}$$
be the set of \textbf{positive roots} of $\NA (V)$ and the \textbf{root system}
$\boldsymbol{\Delta}=\boldsymbol{\Delta}_+ \cup -\boldsymbol{\Delta}_+$ of
$\NA(V)$.
For any $\alpha \in \boldsymbol{\Delta}_+$, the number of elements $u\in L$
such that $\deg(u)=\alpha$ is called the \textbf{multiplicity} of $\alpha $.

Next we recall our fundamental tools in this paper.

For any $i\in I$, let $d_i$ be the skew-derivation of the tensor algebra $T(V)$ given by
$$d_i(x_j)=\delta_{ij},~d_i(xy)=d_i(x)y+\chi(\deg(x),\alpha_i)xd_i(y),$$
for any $j\in I$ and $x,y\in T(V)$.

\begin{rema}
For any homogeneous element $y\in \NA(V)$ of positive degree is zero if and only if $d_i(y)=0$ in $\NA(V)$ for any $i\in I$.
\end{rema}

Let $\ad$ denote the adjoint action of $T(V)$ on itself. Then
$$\ad x_1(y)=x_1y-\chi(\alpha_1,\deg(y))yx_1,~ \text{for any} ~y\in T(V).$$

In the remaining part of this paper we concentrate on the case $n=2$. For the convenience of expressing, let $q=q_{11}, r=q_{12}q_{21}, s=q_{22}$.

For any $k\in\ndN_0$, we define $u_k\in T(V)$ inductively as follows:
  $$u_0=x_2,~~u_k=x_1u_{k-1}-q^{k-1}q_{12}u_{k-1}x_1, ~~\text{for}~~ k\ge1.$$
Note that $u_k=[x_1^kx_2]$.

For any $k\in\ndN_0$, let
  $$b_k=\prod_{j=0}^{k-1}(1-q^jr),$$
and we normalize $u_k$ as follows:
\begin{align*}
  \dpu_k=\begin{cases} \frac{1}{(k)_q^!b_k}u_k & \text{if $(k)_q^!b_k\ne 0$,}\\
	0 & \text{otherwise.}
\end{cases}
\end{align*}

We end this part by recalling some results and some important polynomials in $T(V)$ from \cite{H2018root}.



\begin{lemm}\cite[Lemma~3.9]{H2018root} \label{le:d1ad}
  Let $m\in \ndN_0$. Then
\begin{align*}
  d_1\big( (\ad x_1)^m(y)\big)
  =&\,q^m(\ad x_1)^m(d_1(y))\\
  &\,+(m)_q(1-q^{m-1}\chi(\alpha_1,\alpha)\chi(\alpha,\alpha_1))
  (\ad x_1)^{m-1}(y)
\end{align*}
for any homogeneous element $y\in T(V)$ of degree $\alpha \in \ndN_0^n$.
\end{lemm}

\begin{lemm}\cite[Lemma~3.11]{H2018root} \label{le:dZ}
	Let $k\in \ndN_0$ such that $(k)_q^!b_k\ne 0$ and let
	$\lambda_0,\dots,\lambda_k\in \fie $. Let
	$Z=\sum_{i=0}^k\lambda_i(-q_{21})^i\dpu_i\dpu_{k-i}$ in $T(V)$. Then
	\begin{align*}
	  d_1(Z)=&-q_{21}\sum_{i=0}^{k-1}(\lambda_{i+1}-q^i\lambda_i)(-q_{21})^i
	  \dpu_i\dpu_{k-1-i},\\
		d_2(Z)=&(\lambda_0+(-r)^ks\lambda_k)\dpu_k.
	\end{align*}
\end{lemm}

For any $n\in \ndN _0$ let
\begin{align} \label{eq:Un}
 U_n=&\,\bigoplus_{i=0}^n \fie u_iu_{n-i}\subseteq T(V),&
 U_n'=&\,\bigoplus_{i=0}^{n-1} \fie u_iu_{n-i}\subseteq T(V).
\end{align}

\begin{lemm}\cite[Lemma~4.4]{H2018root}
\label{le:adxinjective}
  The map $\ad x_1:U_m\to U_{m+1}$
  is injective for any $m\in \ndN_0$.
\end{lemm}

\begin{defi}\cite[Definition~4.6]{H2018root} \label{de:Pk}
For all $k\in \ndN_0$ with $(k)_q^!b_k\ne 0$, let
\begin{align*} 
 P_k=\sum_{i=0}^k (-q_{21})^iq^{i(i-1)/2}\dpu_i\dpu_{k-i}\in T(V).
\end{align*}
\end{defi}

\begin{lemm}\cite[Lemma~4.7]{H2018root} \label{d(Pk)}
Let $k\in \ndN_0$ with $(k)_q^!b_k\ne 0$. Then $P_k=0$ in $\NA (V)$ if and only if
$q^{k(k-1)/2}(-r)^ks=-1$.
\end{lemm}

\begin{defi}\cite[Definition~4.8]{H2018root}  \label{de:Skt}
For all $k, t\in \ndN _0$ with $0\le t\le k$ and $(k)_q^!b_k\ne 0$ let
\begin{align*}
	\su(k,t)=& \sum_{i=t}^k(-q_{21})^iq^{(i-t)(i-t-1)/2}{i\choose t}_{\!\!q}
	\dpu_i\dpu_{k-i} \in T(V).
\end{align*}
In particular, $\su(k,0)=P_k$.
\end{defi}

\begin{lemm}\cite[Lemma~4.10]{H2018root} \label{ad pi}
	Let $m,k \in \ndN_0$. Let us assume that $(k+m)_q^!b_{k+m}\ne 0$. Then
\begin{align*}
	&q_{12}^{-m}(\ad x_1)^m(P_k)
	=
	\sum_{i=0}^m \frac{(m)_q^!}{(m-i)_q^!}
	\lambda_{(m-i,k)}
	\beta_{(i,m,k)}\su(k+m,i)
\end{align*}
where for any $i,n,m'\in \ndN_0$,
\begin{align*}\lambda_{(n,k)}=\prod_{j=1}^{n}(1-q^{k-1+j}r)(k+j)_q,\quad
      \beta_{(i,m',k)}=\prod_{j=1}^{i}(q^{m'+2k-j}r-r^{-1}).
		\end{align*}
\end{lemm}

\begin{lemm}\cite[Lemma~4.11]{H2018root} \label{le:411}
Let $k,m\in \ndN_0$. Let us suppose that $(k+m+1)_q^!b_{k+m+1}\ne 0$ and that there exists $v\in
U_{k+m+1}'\cap \ker(\pi)$ such that $d_1(v)=(\ad x_1)^m(P_k)$ in $T(V)$.
Then $Q_2^{k,m}=0$.
\end{lemm}

\section{Root multiplicities over arbitrary fields}
\label{se:multi}
In this section we address the multiplicities for root $m\alpha_1+2\alpha$ with $m\in\ndN_0$ over a field of arbitrary characteristic.

Let $p$ be the characteristic of $\fie$, where $p$ is a prime number or zero.
Now let us introduce the set $\ndJ$ which plays an crucial role in this section.
\begin{defi}\label{def:J}
Let $\ndJ=\ndJ^p_{q,r,s} \subseteq \ndN_0$ be such that $j\in \ndJ$ if and only if
$$q^{j(j-1)/2}(-r)^js=-1,$$ and
for any $n\in \ndJ$ with $n<j$ such that
\begin{align}\label{de:condi}
\begin{cases}
  \left(\frac{j-n}{2}\right)_{\!\!q^{n+j-1}r^2}=0, &~\text{if $j-n$ is even,}\\
 \left(j-n\right)_{-q^{(n+j-1)/2}r}=0, &~\text{if $j-n$ is odd}.
\end{cases}
\end{align}
\end{defi}

\begin{rema}
Note that if $p=0$ then Definition \ref{def:J} is equivalent to \cite[Definition~4.12]{H2018root}.
\end{rema}

\begin{exam}\label{ex:J}
Here we give an example where the $\ndJ$ defined in Definition \ref{def:J} is not the $\ndJ$ defined in \cite[Definition~4.12]{H2018root}.

Assume that $\mathrm{char}(\fie)=3$. Let $s=-1$, $r=-q^{-1}$. Note that then
$0\in \ndJ$. Since $-r\ne0, ~qr^2\ne0$, we get that both $j=1,n=0$ and $j=2,n=0$ don't satisfy \eqref{de:condi}, then $1,2\notin \ndJ$.
As $qr=-1, ~s=-1$ we get $q^{3}(-r)^{3}s=-1$. Moreover we get $(3)_{-qr}=0$, since $qr=-1$ and $p=3$. Hence \eqref{de:condi} holds for $j=3,n=0$, it follows $3\in\ndJ$. Hence $\ndJ\cap[0,3]=\{0,3\}$.

Note that in \cite[Definition~4.12]{H2018root}, $\ndJ\cap[0,3]=\{0\}$.

\end{exam}

\begin{lemm}\label{le:upbound}
For any $j\in\ndJ$, the integers $j+1$ and $j+2$ are not in $\ndJ$. In particular, for any $m\in \ndN_0$,
$$\big|\ndJ\cap[0,m]|\le \frac{m}{3}+1.$$
\end{lemm}
\begin{proof}
Let $j\in \ndN_0$ and $t\in \ndN$. Suppose that $j,j+t\in \ndJ$.
Hence
 $$\left(\frac{t}{2}\right)_{q^{2j+t-1}r^2}=0,~\text{if $t$ is even, or}
  ~(t)_{-q^{(2j+t-1)/2}r}=0, ~\text{if t is odd.}$$
 This gives a contradiction for $t=1$ and $t=2$.
\end{proof}
\begin{exam}\label{ex:smallnumb}
By the definition of $\ndJ$ and by Lemma \ref{le:upbound} the following hold.
\begin{enumerate}
  \item $0\in \ndJ$ if and only if $s=-1$.
  \item $1\in \ndJ$ if and only if $rs=1$ and $s\ne-1$.
  \item $2\in \ndJ$ if and only if $qr^2s=-1$, $rs\ne1$ and $s\ne-1$.
\end{enumerate}
\end{exam}


\begin{lemm}\label{le:J}
Let $j\in \ndN_0$. 
Then $j\in \ndJ$ if and only if one of the following conditions holds:
\begin{enumerate}
\item $q^{j(j-1)/2}(-r)^js=-1$ and $q^{n+j-1}r^2\ne1$, for any $n\in \ndJ$, $n<j$;
\item there exists some $n\in \ndJ$ with $n<j$, such that
\begin{align*}
\begin{cases}
   q^{n+j-1}r^2=1 ~\text{and}~ ~2p|(j-n), &\text{if $j-n$ is even,}\\
   q^{(n+j-1)/2}r=-1 ~\text{and}~ ~p|(j-n),&\text{if $j-n$ is odd}.
\end{cases}
\end{align*}
 \end{enumerate}
\end{lemm}
\begin{proof}



We first prove that for any $j\in\ndN_0$ satisfying condition (1) or (2),  we get
\begin{align}
\label{eq:1}
q^{j(j-1)/2}(-r)^js=-1.
\end{align}
It is clear if $j$ satisfies condition (1).

If $j$ satisfies condition (2), then there exists $n\in\ndJ$ with $n<j$ such that $q^{n+j-1}r^2=1$ and $q^{n(n-1)/2}(-r)^ns=-1$.
If $n-j$ is even, then
$$q^{j(j-1)/2}(-r)^js=-q^{(j-n)(n+j-1)/2}(-r)^{j-n}=-(q^{n+j-1}r^2)^{(j-n)/2}=-1.$$
If $n-j$ is odd, then $q^{(n+j-1)/2}r=-1$,
$$q^{j(j-1)/2}(-r)^js=-q^{(j-n)(n+j-1)/2}(-r)^{j-n}=-(-q^{(n+j-1)/2}r)^{j-n}=-1.$$
Therefore Equation \eqref{eq:1} holds.

Suppose now $j\in \ndN_0$ satisfies Equation \eqref{eq:1}.
For any $n\in\ndJ$ with $n<j$, we have
\begin{align}
\label{eq:2}
q^{n(n-1)/2}(-r)^ns=-1.
\end{align}
From Equations \eqref{eq:1} and \eqref{eq:2} we get, $q^{(j-n)(n+j-1)/2}(-r)^{j-n}-1=0$.

 If $j-n$ is even. Then
 $$0=1-q^{(j-n)(n+j-1)/2}(-r)^{j-n}=(1-q^{n+j-1}r^2)
 \left(\frac{j-n}{2}\right)_{q^{n+j-1}r^2}  $$
Thus $j\in\ndJ$, that is,
  $$\left(\frac{j-n}{2}\right)_{q^{n+j-1}r^2}=0,$$
if and only if either condition (1) or condition (2) holds.

 If $j-n$ is odd,
   $$1-q^{(j-n)(n+j-1)/2}(-r)^{j-n}=(1+q^{(n+j-1)/2}r)
 \left(j-n\right)_{-q^{(n+j-1)/2}r}. $$
 Therefore, $j\in\ndJ$, that is,
 $$\left(n-j\right)_{-q^{(n+j-1)/2}r}=0$$
if and only if either condition (1) or condition (2) holds.

Thus we prove this lemma.
\end{proof}

\begin{rema}
Let $j,m\in \ndN_0$ with $j\le m$. Suppose that $b_m\ne0$, then the condition (2) in Lemma \ref{le:J} is equivalent to the following condition
\begin{align*}
&q^{n+j-1}r^2=1, \text{for some $n\in \ndJ$ with $n<j$, and } \\
&\text{$p|(j-n)$ if $p$ is odd, or $4|(j-n)$ if $p=2$}.
\end{align*}
\end{rema}

\begin{defi}\label{def:J1J2}
Let $\ndJ_1\subseteq\ndJ$ be the set of elements satisfying the condition (1) in Lemma \ref{le:J}.

Let $\ndJ_2\subseteq\ndJ$ be the set of elements satisfying the condition (2) in Lemma \ref{le:J}.
\end{defi}


\begin{rema}\label{re:p=0}
Note that $\ndJ=\ndJ_1\cup\ndJ_2.$
In particular, if $p=0$, then $\ndJ=\ndJ_1$. 
\end{rema}

\begin{exam}
In Example \ref{ex:J}, we obtain that $\ndJ_1\cap[0,3]=\{0\}$ and $\ndJ_2\cap[0,3]=\{3\}$.
\end{exam}

\begin{lemm}\label{le:unique}
Suppose that $n, m\in \ndN_0$ such that $(m)_q^!b_m\ne0$.
Then the following hold.
\begin{enumerate}
\item If $p=2$, then there exists at most one element $j\in\ndN_0$ with $j\le m$ such that
    $ q^{n+j-1}r^2=1.$
\item Assume that $p\ne 2$ and that there exist $j_1,j_2\in\ndN_0$ with $j_1< j_2\le m$ such that
$ q^{n+j_t-1}r^2=1, ~\text{for} ~t=1,2$,
 then $q=1$ and $r=-1$.
Further, we get $\ndJ\cap[0,m]=\{0\}$ or $\ndJ\cap[0,m]=\emptyset$. In particular, $\ndJ_2\cap[0,m]=\emptyset$.
\end{enumerate}
\end{lemm}
\begin{proof}
(1) If $m=0$, the claim is trivial. Suppose that $m\ge1$.
 If there exist $j_1,j_2\in\ndN_0$ such that $j_1<j_2\le m$ and that
 $q^{n+j_i-1}r^2-1=0, ~\text{for} ~i=1,2$. Then $q^{j_2-j_1}=1$. As $(m)_q^!\ne0$ and $j_2-j_1\le m$, we get $q=1$. Therefore $r^2=1$, since $q^{n+j_1-1}r^2-1=0$. Thus $r=1$, since $p=2$, this gives a contradiction to $b_m\ne0$. This proves (1).

(2) From the discussions in (1), it follows that $q=1,r=-1$.

If $\ndJ\ne\emptyset$, then there exists $j\in\ndJ$, such that $q^{j(j-1)/2}(-r)^js=-1$,
the above argument implies $s=-1$. Hence $0\in \ndJ_1$.
Moreover, we obtain that $q^ir^2=1$, for any $i\in \ndN_0$, thus $\ndJ_1=\{0\}$.

If $p=0$, then it is clear that $\ndJ=\ndJ_1=\{0\}$, applying Remark ~\ref{re:p=0}.
Assume now $p$ is positive. Let us prove $\ndJ_2\cap[0,m]=\emptyset$.
Suppose that $\ndJ_2\cap[0,m]\ne\emptyset$, and let $n\in\ndN$ be the minimal element in $\ndJ_2\cap[0,m]$.
From Lemma \ref{le:J} and $\ndJ_1=\{0\}$ we get that $p|n$. Thus $p\le n\le m$ and $(p)_q=0$ as $q=1$, this is contradictory to $(m)_q^!\ne0$. Thus $\ndJ_2\cap[0,m]=\emptyset$. Hence
$\ndJ=\{0\}$.
\end{proof}

\begin{lemm}\label{le:J2}
Suppose that $m\in \ndN_0$ and that $(m)_q^!b_m\ne0$. Then
\begin{enumerate}
 \item $j\in \ndJ_1\cap[0,m]$ if and only if
    $$ q^{j(j-1)/2}(-r)^js=-1,\text{and}~ q^{n+j-1}r^2\ne1~\text{for any $n\in\ndJ_1$ $n<j$}.$$

 \item $j\in \ndJ_2\cap[0,m]$ if and only if there exists $n\in\ndJ_1$ with $n<j$ such that
   $$q^{n+j-1}r^2=1 ~\text{and $p|(j-n)$ if $p$ is odd, or $4|(j-n)$ if $p=2$}.$$
\end{enumerate}
\end{lemm}

\begin{proof}
It is enough to prove that for any $j\in[0,m]$, we have
  $$q^{n+j-1}r^2\ne1,~ \text{for all $n\in\ndJ_2$ with $n<j$}.$$

If $\ndJ_2\cap[0,m]=\emptyset$ there is nothing to prove.
Let us suppose now $\ndJ_2\cap[0,m]\ne\emptyset$.
Assume oppositely that $j\in[0,m]$ and that there exists $n\in\ndJ_2$ with $n<j$ such that
 $q^{n+j-1}r^2=1$. As $n\in \ndJ_2$, let $k\in\ndJ$ with $k<n$ such that
 $q^{n+k-1}r^2=1$. As $k<j\le m$, we get an contradiction when $p=2$ because of Lemma \ref{le:unique}(1). When $p\ne2$, again as $k<j\le m$,
 it follows that $\ndJ_2\cap[0,m]=\emptyset$ because of Lemma \ref{le:unique}(2), this is also an contradiction. Hence the lemma follows.
\end{proof}

\begin{lemm}\label{le:unique2}
Suppose that $m\in \ndN_0$, and that $(m)_q^!b_m\ne0$, then for any $n\in \ndJ_2\cap[0,m]$, there exists a unique $j_n\in \ndJ$ with $j_n<n$ such that $q^{n+j_n-1}r^2-1=0$. Moreover, $j_n\in \ndJ_1$.
\end{lemm}
\begin{proof}
Assume that $n\in \ndJ_2$.
From the Lemma \ref{le:J} and Definition~\ref{def:J1J2}, we know that there exists $j_n\in \ndJ$ with $j_n<n$ such that
$q^{n+j_n-1}r^2-1=0$.
If there exists $ j_n'\in \ndJ$, with $n>j_n>j_n'$, satisfies $q^{n+j_n'-1}r^2-1=0$. Again using Lemma~\ref{le:unique} we obtain that $\ndJ_2\cap[0,m]=\emptyset$,
this is a contradiction.

Clearly, $j_n\in \ndJ_1$ by applying Lemma~\ref{le:J2}(2).
\end{proof}


\begin{lemm}\label{le:exist}
Let $k\in \ndN$.
Suppose that
$$w=\sum_{i=0}^{k-1}\mu_i(-q_{21})^i\dpu_i\dpu_{k-1-i},~\text{and}~~
\sum_{i=0}^{k-1}q^{-i(i+1)/2}\mu_i=0.$$
Then there exist a unique element  $v=\sum_{i=1}^{k-1}\lambda_i(-q_{21})^i\dpu_i\dpu_{k-i}\in U_k'$
such that
 $-q_{21}^{-1}d_1(v)=w.$
\end{lemm}
\begin{proof}
Assume that $v=\sum_{i=0}^{k}\lambda_i(-q_{21})^i\dpu_i\dpu_{k-i}\in U_k',
~\lambda_0=0, ~\lambda_k=0$, then
\begin{align*}
-q_{21}^{-1}d_1(v)=\sum_{i=0}^{k-1}(\lambda_{i+1}-q^i\lambda_i)(-q_{21})^i\dpu_i\dpu_{k-1-i},
\end{align*}
because of Lemma~\ref{le:dZ}.
Thus
 $-q_{21}^{-1}d_1(v)=w$ if and only if
  $$\mu_i=\lambda_{i+1}-q^i\lambda_i, ~\text{for}~0\le i\le k-1.$$
 Hence
  $$\lambda_i=\sum_{j=0}^{i-1}q^{(i+j)(i-j-1)/2}\mu_j,~\text{for all $0\le i\le k$}.$$

We still need to prove that $\lambda_k=0$. While $\sum_{i=0}^{k-1}q^{-i(i+1)/2}\mu_i=0$ implies that
$$\lambda_k=\sum_{j=0}^{k-1}q^{(k+j)(k-j-1)/2}\mu_i
=q^{k(k-1)}\sum_{j=0}^{k-1}q^{-j(j+1)/2}\mu_i=0.$$

The uniqueness is clear. Thus we prove this lemma.
\end{proof}

We recall the definition of $Q_2^{k,m}\in\ndZ[q,r]$ from \cite[Lemma~2.3]{H2018root},
$$Q_2^{k,m}=
 \frac{q^{(2k+m)(m+1)/2}(-r)^{m+1}-1}{q^{2k+m}r^2-1}
 \prod_{i=0}^{m}(1-q^{k+i}r).$$
In this paper we view $Q_2^{k,m}$ as element in $\fie=\fie\ot_{\ndZ[q,r]}\ndZ[q,r]$ by identifying $q$ and $r$ in $\ndZ[q,r]$ with $q$ and $r$ in $\fie$, respectively.
\begin{lemm}\label{le:Q2vanish}
Assume that $p$ is a prime number. Let $j_n,n\in \ndN_0$ with $j_n<n$ such that $b_n\ne0$. Suppose that $j_n\in\ndJ_1$ and $q^{n+j_n-1}r^2=1$, then $n\in \ndJ_2$ if and only if $Q_2^{j_n,n-j_n-1}=0$.
\end{lemm}

\begin{proof}
Assume that $q^{n+j_n-1}r^2=1$.

 If $n-j_n$ is even, then
 \begin{align*}
		Q_2^{j_n,n-j_n-1}&=\sum_{i=0}^{(n-j_n-2)/2}(q^{n+j_n-1}r^2)^i
		\prod_{i=0}^{n-j_n-1}(1-q^{j_n+i}r)\\
        &=\frac{n-j_n}2\prod_{i=0}^{n-j_n-1}(1-q^{j_n+i}r).
	\end{align*}
Since $b_n\ne0$, we get $Q_2^{j_n,n-j_n-1}=0$ if and only if $2p|(n-j_n)$, that is, $n\in \ndJ_2$ applying Lemma \ref{le:J}.

If $n-j_n$ is odd. Since $b_n\ne0$ we get $q^{(n+j_n-1)/2}r=-1$. Let us write $k=(n-j_n-1)/2$, then $q^{j_n+k}r=-1$. Hence
\begin{align*}
	  Q_2^{j_n,n-j_n-1}&=\sum_{i=0}^{n-j_n-1}(-q^{j_n+k}r)^i
		\prod_{i=0}^{k-1}(1-q^{j_n+i}r)\prod_{i=k+1}^{2k}(1-q^{j_n+i}r)\\
     &=(n-j_n)\prod_{i=0}^{k-1}(1-q^{j_n+i}r)\prod_{i=k+1}^{2k}(1-q^{j_n+i}r)
	\end{align*}
Again since $b_n\ne0$, we conclude that $Q_2^{j_n,n-j_n-1}=0$ if and only if $p|(n-j_n)$, again we get $n\in \ndJ_2$.

Thus we prove this lemma.
\end{proof}

\begin{lemm}\label{le:existence}
Assume that $p$ is a prime number.
Let $n\in \ndN_0$ such that that $(n)_q^!b_n\ne0$. Suppose that $n\in \ndJ_2$, and that $j_n\in \ndJ_1$ with $j_n<n$ such that $q^{n+j_n-1}r^2=1$.
Then there exits a unique element $L_n\in U'_{n}\cap \ker(\pi)$ such that $$-q_{21}^{-1}d_1(L_n)=(\ad x_1)^{n-j_n-1}(P_{j_n}).$$
\end{lemm}

\begin{proof}
Assume that there exists an element $L_n\in U'_{n}$ such that
$$-q_{21}^{-1}d_1(L_n)=(\ad x_1)^{n-j_n-1}(P_{j_n}),$$
it follows from $d_2(L_n)=0$ and
$(\ad x_1)^{n-j_n-1}(P_{j_n})=0$ in $\NA(V)$ that $L_n\in \ker(\pi)$.

Let now $(\ad x_1)^{n-j_n-1}(P_{j_n})=\sum_{i=0}^{n-1}\mu_i(-q_{21})^i\dpu_i\dpu_{n-1-i}$.
It is enough to prove $\sum_{i=0}^{n-1}q^{-i(i+1)/2}\mu_i=0$ by Lemma \ref{le:exist}. Let us write $m=n-j_n-1$ and $H=\sum_{i=0}^{n-1}q^{-i(i+1)/2}\mu_i$.
Then
\begin{align*}
&H
 =q_{12}^{m}\sum_{i=0}^{m} \frac{(m)_q^!}{(m-i)_q^!}
\lambda_{(m-i,j_n)} \beta_{(i,m,j_n)}
\sum_{j=i}^{n-1}q^{(j-i)(j-i-1)/2}q^{-j(j+1)/2}{j\choose i}_{\!\!q}
\end{align*}



By applying the similar discussions in the proof of \cite[Lemma ~4.11]{H2018root} for $m=n-j_n-1$, $k=j_n$,  we conclude that
$$H=q_{12}^{m}q^{-(n-j_n)(n+j_n-1)/2}{n\choose j_n}_{\!\!q}(-r)^{-m}Q_2^{j_n,n-j_n-1}$$
Thus $H=0$ by Lemma \ref{le:Q2vanish}.
\end{proof}

\begin{prop}\label{prop:linindep}
Assume that $m\in \ndN_0$ and that $(m)_q^!b_m\ne0$. The elements
$$(\ad x_1)^{m-j}(P_j) ~\text{for} ~j\in \ndJ_1$$
are linearly independent in $\ker(\pi)\cap U_m$.
\end{prop}
\begin{proof}
It was proved in \cite[Theorem 4.16]{H2018root}.
\end{proof}

\begin{prop}\label{prop:independ}
Suppose that $m\in \ndN_0$ and that $(m)_q^!b_m\ne0$. The elements
$$(\ad x_1)^{m-j}(P_j), ~(\ad x_1)^{m-n}(L_n),~ \text{for} ~j\in \ndJ_1\cap[0,m] ~\text{and}~ n\in \ndJ_2\cap[0,m],$$
are linearly independent in $\ker(\pi)\cap U_m$.
\end{prop}
\begin{proof}
If $p=0$, then $\ndJ_2=\emptyset$, the claim follows from Proposition \ref{prop:linindep}.

Assume now $p$ is a prime number.
We proceed this claim by induction on $m$.

This is clear for $m=0$.
Assume now that $m\ge1$ and for any $j\in \ndJ\cap[0,m]$ there is
$\lambda_j\in\fie$ such that
$$\sum_{n\in \ndJ_2\cap[0,m]}\lambda_{n}(\ad x_1)^{m-n}(L_n)+
\sum_{j\in \ndJ_1\cap[0,m]}\lambda_{j}(\ad x_1)^{m-j}(P_j)=0.$$

If $m\notin \ndJ_1\cup\ndJ_2$, then
$$\sum_{n\in \ndJ_2\cap[0,m-1]}\lambda_{n}(\ad x_1)^{m-1-n}(L_n)+
\sum_{j\in \ndJ_1\cap[0,m-1]}\lambda_{j}(\ad x_1)^{m-1-j}(P_j)=0,$$
as $\ad x_1|_{U_{m-1}}$ is injective by Lemma~\ref{le:adxinjective}.
By induction hypothesis, we obtain that $\lambda_{j}=0$ for any $j\in \ndJ\cap[0,m]$.

Assume now $m\in \ndJ_1$, then $m\notin \ndJ_2$ and $d_1(P_m)=0$.
And then
$$\sum_{n\in \ndJ_2\cap[0,m-1]}\lambda_{n}d_1((\ad x_1)^{m-n}(L_n))+
\sum_{j\in \ndJ_1\cap[0,m-1]}\lambda_{j}d_1((\ad x_1)^{m-j}(P_j))=0.$$

From Lemma \ref{le:d1ad} we have that,
\begin{align*}
&\sum_{n\in \ndJ_2\cap[0,m-1]}\lambda_{n}(-q_{21})q^{m-n}(\ad x_1)^{m-1-j_n}(P_{j_n})\\
 &  + \sum_{n\in \ndJ_2\cap[0,m-1]}\lambda_{n}(m-n)_q(1-q^{m+n-1}r^2)(\ad x_1)^{m-1-n}(L_n)\\
&+\sum_{j\in \ndJ_1\cap[0,m-1]}\lambda_{j}(m-j)_q(1-q^{m+j-1}r^2)(\ad x_1)^{m-1-j}(P_j)=0.
\end{align*}
 Lemma~\ref{le:J} and $m\in \ndJ_1$ imply that $(m-n)_q(1-q^{m+n-1}r^2)\ne0$ for any $n\in \ndJ_2\cap[0,m-1]$, by induction hypothesis we obtain that
$\lambda_{n}=0$, for any $n\in \ndJ_2\cap[0,m-1]$.
Then $\lambda_{j}=0$ for any $j\in\ndJ_1$ because of Proposition~\ref{prop:linindep}.

If $m\in \ndJ_2$, then $m\notin \ndJ_1$ and then there exists a unique $j_m\in \ndJ_1\cap[0,m]$ such that $q^{m+j_m-1}r^2-1=0$,
then

\begin{align}
\label{eq:3}
&\sum_{n\in \ndJ_2\cap[0,m]}\lambda_{n}(-q_{21})q^{m-n}(\ad x_1)^{m-1-j_n}(P_{j_n})\\
\label{eq:4}
 &+ \sum_{n\in \ndJ_2\cap[0,m-1]}\lambda_{n}(m-n)_q(1-q^{m+n-1}r^2)
 (\ad x_1)^{m-1-n}(L_n)\\
&+\sum_{j\in \ndJ_1\cap[0,m-1],j\ne j_m}\lambda_{j}(m-j)_q(1-q^{m+j-1}r^2)(\ad x_1)^{m-1-j}(P_j)=0
\end{align}

Due to Lemma~\ref{le:unique2}, we get that $(m-n)_q(1-q^{m+n-1}r^2)\ne0$ for any $n\in \ndJ_2\cap[0,m-1]$. Then
$\lambda_{n}=0$, for any $n\in \ndJ_2\cap[0,m-1]$ by induction hypothesis. This forces formula \eqref{eq:3} to be $-q_{21}\lambda_m(\ad x_1)^{m-1-j_m}(P_{j_m})$ and formula \eqref{eq:4} to be zero. Further more,
Lemma~\ref{le:unique2} and $(m)_q^!\ne0$ imply that $(m-j)_q(1-q^{m+j-1}r^2)\ne0$ for any $j\in \ndJ_1\cap [0,m]$ and $j\ne j_m$.
It follows from Proposition \ref{prop:linindep} that
$\lambda_j=0$ for any $j\in \ndJ_1\cap[0,m]$ with $j\ne j_m$ and $\lambda_{m}=0$. Therefore $\lambda_j=0$ for any $j\in \ndJ\cap[0,m]$.

This finishes the proof of this proposition.
\end{proof}

\begin{theo}\label{theo:main}
Suppose that $m\in \ndN_0$ and that $(m)_q^!b_m\ne0$. The elements
$$(\ad x_1)^{m-j}(P_j), ~(\ad x_1)^{m-n}(L_n),~ \text{for} ~j\in \ndJ_1 ~\text{and}~ n\in \ndJ_2,$$
form a basis of $\ker(\pi)\cap U_m$.
\end{theo}

\begin{proof}
Assume first that $p=0$, then $\ndJ_2=\emptyset$ and the claim was proved in
\cite[Theorem~4.16]{H2018root}.

Assume now $p$ is a prime number.
It follows from Proposition \ref{prop:independ} that
it is enough to prove
\begin{align}
\label{eq:size}
\dim(\ker(\pi)\cap U_m)=|\ndJ\cap[0,m]|.
\end{align}
Again we proceed by induction on $m$. It is clear that $\ndJ_2=\emptyset$, $\ndJ=\ndJ_1$ for $m=0$. Since $P_0=u_0^2=0$ in $\NA(V)$ if and only if $1+s=0$, that is, $0\in\ndJ$. Thus the claim holds for $m=0$.

Suppose now $m>0$. Applying induction hypothesis and Proposition \ref{prop:independ}, we obtain that
the elements $(\ad x_1)^{m-1-j}(P_j), ~(\ad x_1)^{m-n-1}(L_n)$, for $j\in \ndJ_1\cap [0,m-1] ~\text{and}~ n\in \ndJ_2\cap[0,m-1]$ form a basis of $\ker(\pi)\cap U_{m-1}$.

Since $\ad x_1(\ker(\pi))\subseteq\ker(\pi)$, and $\ad x_1$ is injective by Lemma ~\ref{le:adxinjective} it follows that
$$\dim (\ker(\pi)\cap U_m)\ge \dim (\ker(\pi)\cap U_{m-1}).$$

Suppose first that $\dim (\ker(\pi)\cap U_m)=\dim (\ker(\pi)\cap U_{m-1}).$
Then the elements  $(\ad x_1)^{m-j}(P_j), ~(\ad x_1)^{m-n}(L_n)$, for $j\in \ndJ_1\cap [0,m-1] ~\text{and}$ $~n\in \ndJ_2\cap[0,m-1]$ form a basis of $\ker(\pi)\cap[0,m]$. And Proposition \ref{prop:independ} implies that $m\notin \ndJ$. Thus \eqref{eq:size} follows.

Assume now $\dim (\ker(\pi)\cap U_m)>\dim (\ker(\pi)\cap U_{m-1}).$  Since
$$d_1(\ker(\pi)\cap U_m)\subseteq \ker(\pi)\cap U_{m-1},$$
we obtain that $\ker(\pi)\cap U_m\cap \ker(d_1)\ne0$.
As $(m)_q^!b_m\ne0$, it follows that $\ker(d_1|U_m)=\fie P_m$. Hence
\begin{align*}
& \dim (\ker(\pi)\cap U_m)=1+\dim (\ker(\pi)\cap U_{m-1}),\\
 & q^{m(m-1)/2}(-r)^ms=-1.
\end{align*}
Therefore $d_1|\ker(\pi)\cap U_m$ is surjective. Thus for any $j\in\ndJ_1\cap[0,m-1]$, there exists $v_j\in \ker(\pi)\cap U_m$ such that
 $$d_1(v_j)=(\ad x_1)^{m-1-j}(P_j).$$
Further, we may assume that $v_j\in \ker(\pi)\cap U_m'$ for any $j\in \ndJ_1\cap [0,m-1]$.
Then $Q_2^{j,m-j-1}=0$ for any $j\in\ndJ_1\cap [0,m-1]$ by Lemma~\ref{le:411}. If
$q^{m+j-1}r^2\ne1$
 for any $j\in\ndJ_1\cap [0,m-1]$, then $m\in \ndJ_1$ follows from Lemma~\ref{le:J2}(1).
 If
 $q^{m+j-1}r^2=1$
 for some $j\in\ndJ_1\cap [0,m-1]$. Since $b_m\ne0$, Lemma~\ref{le:Q2vanish} implies that $m\in \ndJ_2$.
Thus Equation \eqref{eq:size} holds.
\end{proof}

\begin{rema}
If the $p=0$, then  Theorem \ref{theo:main} is exactly  Theorem
4.16 in \cite{H2018root}. It can happen that $\ndJ_2=\emptyset$, even if $p$ is prime number.
\end{rema}

\begin{rema}
 Theorem \ref{theo:main} generalizes Theorem $4.16$ in \cite{H2018root} to the field of arbitrary characteristic.
 As a consequence, we generalize Corollaries $4.17$, $4.18$ 4.22 and Proposition $4.20$ in \cite{H2018root} to arbitrary fields
\end{rema}

\begin{coro} \label{co:rootvectors}
 Let $k,l\in \ndN_0$
  with $k\ge l$. Suppose that $(k+l)_q^!b_{k+l}\ne 0$, and that
  $q^{k^2}r^ks=-1$ if $k=l$.
  Then the following are equivalent:
	\begin{enumerate}
		\item $[x_1^kx_2x_1^lx_2]$ is a root vector,
    \item $\big|\ndJ\cap [0,k+l]\big|\le l$.
	\end{enumerate}
\end{coro}
\begin{proof}
This claim can be obtained by applying the similar approach in the proof of \cite[Corollary~4.17]{H2018root}.
\end{proof}

\begin{coro} \label{size}
Let $m\in \ndN_0$
	such that $(m)_q^!b_m\ne 0$.
Then the multiplicity of $m\alpha_1+2\alpha_2$ is
$$ m'-\big|\ndJ\cap[0,m]\big|,$$
where
\begin{align*}
  m'=\begin{cases} (m+1)/2 & \text{if $m$ is odd,}\\
  m/2 & \text{if $m$ is even and $q^{m^2/4}r^{m/2}s\ne -1$,}\\
  m/2+1 & \text{if $m$ is even and $q^{m^2/4}r^{m/2}s=-1$.}
\end{cases}
\end{align*}
\end{coro}
\begin{proof}
One can prove the claim by using the similar discussions in the proof of \cite[Corollary~4.18]{H2018root}.
\end{proof}

\begin{table}[!hbp]
\begin{tabular}{c|l}
$m\alpha_1+2\alpha_2$ & non-root conditions \\
\hline
$\alpha_1+2\alpha_2$ &   $(1+s)(1-rs)=0$ \\
\hline
$2\alpha_1+2\alpha_2$ &$(1+s)(1-rs)(1+qr^2s)=0$  \\
\hline
$3\alpha_1+2\alpha_2$ & $s=-1,(3)_{-qr}=0$\\
 \hline
 $4\alpha_1+2\alpha_2$ & \tabincell{l}{ $s=-1$, $(3)_{-qr}=0$ or \\
                                        $s=-1$, $q^3r^2=-1$ or\\
                                        $rs=1$, $(3)_{-q^2r}=0$}\\
\hline
$6\alpha_1+2\alpha_2$ & $q=1,s=-1$, $(3)_{-r}=0$
\end{tabular}
\caption{Table for Corollary~\ref{co:non-roots}}
\label{ta:noroot}
\end{table}
\begin{prop} \label{pr:refl}
 Assume that $k,m\in \ndN_0$ with $m\ge k$
	such that $(k)_q^!b_k\ne 0$, and~$(k+1)_q(1-q^kr)=0$.
	Then the multiplicity of $m\alpha_1+2\alpha_2$ is
	the same as the multiplicity of $(2k-m)\alpha_1+2\alpha_2$ of $\NA (R_1(V))$.
\end{prop}

\begin{coro} \label{co:non-roots}
Let $m\in \ndN_0$.
\begin{enumerate}
\item Assume that $m\in \{1,2,3,4,6\}$ and that $(m)_q^!b_m\ne 0$.
Then $m\alpha_1+2\alpha_2$ is not a root if and only if
$q,r,s$ satisfy the conditions given in Table~\ref{ta:noroot}.
\item Assume that $m=2k+1\ge 5$ is odd and that $(k+3)_q^!b_{k+3}\ne0$.
Then $m\alpha_1+2\alpha_2$ is a root of $\NA (V)$.
\item Assume that $m=2k\ge 8$ and that $(k+4)_q^!b_{k+4}\ne 0$.
Then $m\alpha_1+2\alpha_2$ is a root of $\NA (V)$.
\end{enumerate}
\end{coro}

\begin{proof}
Due to Lemma ~\ref{le:upbound} and Corollary~\ref{size}, statements (1), (2), (3) can be obtained similarly to the proofs of statements (1), (2), (3) in
\cite[Corollary 4.22]{H2018root}, respectively.
\end{proof}


\begin{thebibliography}{1}

\bibitem{AS98}
N. Andruskiewitsch and H-J. Schneider.
\newblock Lifting of quantum linear spaces and pointed Hopf algebras of order $p^3$.
\newblock {\em J. Algebra } 209.2: 658-691, 1998.

\bibitem{H2006}
I.~Heckenberger.
\newblock The {W}eyl groupoid of a {N}ichols algebra of diagonal type.
\newblock {\em Invent. {M}ath.}, 164(1):175--188, 2006.

\bibitem{H08}
I. Heckenberger.
\newblock Rank 2 Nichols algebras with finite arithmetic root system.
\newblock {\em Algebra. Represent. Theory} 11.2: 115-132,2008.

\bibitem{H2018root}
I.~Heckenberger and Y.~Zheng.
\newblock Root multiplicities for {N}ichols algebras of diagonal type of rank
  two.
\newblock {\em J. Algebra}, 496:91--115, 2018.


\bibitem{khar}
V. K. Kharchenko.
\newblock A quantum analog of the {P}oincar{\'e}-{B}irkhoff-{W}itt theorem.
\newblock {\em Alg. Logic}, 38(4):259--276, 1999.

\bibitem{ML97}
M. Lothaire.
\newblock Combinatorics on words.
\newblock {\em Cambridge university press, Vol. $17.$}, 1997.


\bibitem{MR1632802}
M. Rosso.
\newblock Quantum groups and quantum shuffles.
\newblock {\em Invent. Math.}, 133(2):399--416, 1998.

\bibitem{MR0506406}
W.~D. Nichols.
\newblock Bialgebras of type one.
\newblock {\em Comm. Algebra}, 6(15):1521--1552, 1978.




\end{thebibliography}
\end{document}